\documentclass[11pt]{amsart}
\usepackage{amssymb,amsmath,mathrsfs}
\usepackage{bbold}

 \newcommand{\fs}{\mathop{\mathrm{FS}}}
 \newcommand{\fu}{\mathop{\mathrm{FU}}}
 \newcommand{\fp}{\mathop{\mathrm{FP}}}
 \newcommand{\supp}{\mathop{\mathrm{supp}}}
 \newcommand{\cf}{\mathop{\mathrm{cf}}}
 \newcommand{\hind}[1]{\mathnormal{\mathsf{HIND}(#1)}}
 
 \newcommand{\zfc}{\mathnormal{\mathsf{ZFC}}}
 \newcommand{\ch}{\mathnormal{\mathsf{CH}}}

\newtheorem{theorem}{Theorem}
\newtheorem*{claim}{Claim}
\newtheorem{question}{Question}
\theoremstyle{definition}
\newtheorem{example}{Example}

\title{Hindman's Theorem is only a Countable Phenomenon}
\thanks{The author was partially supported by postdoctoral fellowship number 263820 from the Consejo Nacional de Ciencia y Tecnolog\'{\i}a (Conacyt), Mexico.}

\author[D. Fern\'andez]{David~J. Fern\'andez-Bret\'on}

\address{
              Department of Mathematics, University of Michigan \\
	      2074 East Hall, 530 Church Street \\
	      Ann Arbor, MI 48109-1043, U.S.A.}
\email{djfernan@umich.edu}
\urladdr{http://www-personal.umich.edu/\textasciitilde djfernan/}



\begin{document}

\begin{abstract}
We pursue the idea of generalizing Hindman's Theorem to uncountable 
cardinalities, by analogy with the way in which Ramsey's Theorem can be 
generalized to weakly compact cardinals. But unlike Ramsey's Theorem, 
the outcome of this paper is that the natural generalizations of Hindman's 
Theorem proposed here tend to fail at all uncountable cardinals.
\end{abstract}

\maketitle

\section{Introduction}
\label{intro}

Hindman's theorem is one of the most famous and interesting examples of a so-called 
\emph{Ramsey-type theorem}, a theorem about partitions.

\begin{theorem}[Hindman \cite{hindmanthm}]\label{eldehindman}
 For every partition $\mathbb N=A_0\cup A_1$ of the set of natural numbers into two cells, 
 there exists 
 an infinite $X\subseteq\mathbb N$ such that for some $i\in2$, $\fs(X)\subseteq A_i$ (where $\fs(X)$ 
 denotes the set
\begin{equation*}
\left\{\sum_{x\in F}x\bigg|F\subseteq X\mathrm{\ is\ finite\ and\ nonempty}\right\}
\end{equation*}
of all \emph{finite sums} of elements of $X$).
\end{theorem}

Hindman's original proof of Theorem~\ref{eldehindman} is long and involved, but a much simpler 
proof, due to Baumgartner, can be found in~\cite{baumgartner}. Both proofs make extensive use of the fact that 
the statement of Theorem~\ref{eldehindman} is equivalent to the following statement.

\begin{theorem}\label{eldebaumgartner}
 For every partition $[\mathbb N]^{<\aleph_0}=A_0\cup A_1$ of the collection $[\mathbb N]^{<\aleph_0}$ 
 of finite subsets of $\mathbb N$ into two cells, there exists an infinite family 
 $X\subseteq[\mathbb N]^{<\aleph_0}$ of pairwise disjoint finite subsets of $\mathbb N$ and an 
 $i\in2$ such that $\fu(X)\subseteq A_i$, where $\fu(X)$ denotes the set 
 \begin{equation*}
 \left\{\bigcup_{Y\in F}Y\bigg|F\subseteq X\mathrm{\ is\ finite\ and\ nonempty}\right\}
 \end{equation*}
of \emph{finite unions} of elements from $X$.
\end{theorem}

The equivalence between Theorems~\ref{eldehindman} and ~\ref{eldebaumgartner} follows 
from~\cite[Lemma 2.2]{hindmanthm} after identifying natural numbers with the support of 
their binary expansion (the (finite) set of places where the corresponding digit in the binary 
expansion of the given number is nonzero), and this fact is explicitly pointed out in~\cite[p. 384]{baumgartner}. 
Unlike Theorem~\ref{eldehindman}, whose statement relies on a very specific semigroup 
operation on the set $\mathbb N$, the statement of Theorem~\ref{eldebaumgartner} seems 
easily adaptable to higher cardinalities. So if $\kappa,\lambda$ are cardinal numbers 
with $\lambda\leq\kappa$, we will 
denote by $\hind{\kappa,\lambda}$ the statement asserting that for every partition 
$[\kappa]^{<\aleph_0}=A_0\cup A_1$ of the family $[\kappa]^{<\aleph_0}$ of finite 
subsets of $\kappa$ into two cells, there exists a family $X\subseteq[\kappa]^{<\aleph_0}$ 
of cardinality $\lambda$, consisting of pairwise disjoint (finite) subsets of $\kappa$, 
such that for some $i\in2$ we have that $\fu(X)\subseteq A_i$. With this notation, the 
``finite-unions" version of Hindman's theorem simply states that $\hind{\aleph_0,\aleph_0}$ 
(and hence also 
$\hind{\kappa,\aleph_0}$ for any infinite $\kappa$) holds. This notation intends to provide an 
analogy with the arrow notation used for generalizations of Ramsey's theorem, where  
$\kappa\longrightarrow(\lambda)_2^2$ 
denotes that for every partition $[\kappa]^2=A_0\cup A_1$ of the family of unordered pairs 
of $\kappa$ into two cells, there exists a subset $X\subseteq\kappa$ of cardinality $\lambda$ such 
that for some $i\in2$ we have that $[X]^2\subseteq A_i$. Thus Ramsey's theorem states 
that $\aleph_0\longrightarrow(\aleph_0)_2^2$, and for uncountable $\lambda$, the existence 
of a $\kappa$ such that $\kappa\longrightarrow(\lambda)_2^2$ follows from the Erd\H{o}s-Rado 
Theorem. However, the existence of an uncountable cardinal $\kappa$ such that 
$\kappa\longrightarrow(\kappa)_2^2$ is not provable in $\zfc$ (the usual axioms of 
mathematics), since such a $\kappa$ would be what is known as a weakly compact cardinal 
(which is a notion of large cardinal and hence its existence goes beyond the $\zfc$ 
axioms). Similarly, we 
might try to address the question of whether we can have $\hind{\kappa,\lambda}$ hold for 
some uncountable cardinals $\kappa,\lambda$ under certain conditions (with perhaps some large cardinal 
assumptions being the primary culprits).

Although this question seems (at least to the author) to be a very natural one, there seems 
to be very few earlier results along these lines. A theorem \cite[Th. 9]{milliken} of Milliken  
is easily seen to establish that $\hind{\kappa^+,\kappa^+}$ fails whenever $\kappa$ is an infinite 
cardinal such that $2^\kappa=\kappa^+$ (in particular, the Continuum Hypothesis $\ch$ implies that 
$\hind{\aleph_1,\aleph_1}$ fails). And the author was recently made aware \cite{michael-osvaldo} of 
an argument of Moore that shows that for every uncountable cardinal $\kappa$, the 
statement $\hind{\kappa,\kappa}$ implies that $\kappa\longrightarrow(\kappa)_2^2$ and therefore 
$\hind{\kappa,\kappa}$ fails unless $\kappa$ is a weakly compact cardinal (in particular, 
$\hind{\aleph_1,\aleph_1}$ fails even without assuming $\ch$). The main result from 
this paper is that the statement $\hind{\kappa,\lambda}$ 
fails whenever $\aleph_0<\lambda\leq\kappa$, thus banishing every possible attempt of generalizing 
Hindman's theorem, at least in its finite-unions version, to uncountable cardinalities. We also 
state other possible generalizations of Hindman's theorem to the uncountable realm, and show 
that they all fail as well.

It is 
worth noting that the generalizations considered in this paper are of a quantitative nature,
meaning that they deal
only with cardinality, hence the results obtained do not preclude the possibility of obtaining 
other sorts of generalizations of Hindman's Theorem. For example, Tsaban~\cite{tsaban} has 
shown that Hindman's Theorem
may be viewed as a colouring theorem dealing with open covers of a certain
countable topological
space, and he proved a generalization of this theorem to arbitrary Menger
spaces (which can have any arbitrary cardinality, although the objects 
coloured in this result are countable families of open sets). As another example, 
Zheng~\cite{y2zheng} has built on Todor\v{c}evi\'c's theory of Ramsey 
Spaces~\cite{ramseysp} to obtain results where finite subsets of $\omega$ together 
with real numbers are coloured, and monochromatic combinations of $\fu$-sets and 
perfect subsets of $\mathbb R$ 
are obtained. Thus, it is possible 
to improve Hindman's theorem in terms of the richness of structure of the obtained 
monochromatic structure, but not in terms of its size.

The second section of this paper contains the proof that $\hind{\kappa,\lambda}$ 
fails for every uncountable $\lambda\leq\kappa$. In the 
third and fourth sections, we consider other possible generalizations of Hindman's 
Theorem to groups and semigroups, and we explain that most of these fail as 
well.\footnote{Further results along these lines are 
obtained in a recent follow-up joint paper of the author and Rinot~\cite{assaf-david}.}

\section{The Uncountable Hindman Statement Fails}
\label{sec:1}

The following argument (which, quite surprisingly, is reasonably simple) establishes the main 
result of this paper.

\begin{theorem}\label{nouncountablehindman}
 Let $\kappa$ be any infinite cardinal. Then there is a partition of $[\kappa]^{<\omega}$ into 
 two cells, none of which can contain $\fu(X)$ for any uncountable pairwise disjoint 
 family $X\subseteq[\kappa]^{<\omega}$. Thus, for any uncountable $\lambda$ and any $\kappa\geq\lambda$, 
 the statement $\hind{\kappa,\lambda}$ fails.
\end{theorem}

\begin{proof}
 Partition $[\kappa]^{<\omega}$ as $A_1\cup A_0$, where
 \begin{equation*}
 A_i=\{x\in[\kappa]^{<\omega}\big|\lfloor\log_2|x|\rfloor\equiv i\mod 2\},
 \end{equation*}
 this is, $x\in A_i$ if and only if the unique $k\in\omega$ such that $2^k\leq|x|<2^{k+1}$ 
 has the same parity as $i$.
 Note that, if $x$ and $y$ are disjoint, both in the same $A_i$, and 
 $\lfloor\log_2|x|\rfloor=\lfloor\log_2|y|\rfloor=k$, then (since $|x\cup y|=|x|+|y|$) 
 we have that $\lfloor\log_2|x\cup y|\rfloor=k+1$ and hence $x\cup y\in A_{1-i}$. Thus if 
 $X\subseteq[\kappa]^{<\omega}$ is a 
 pairwise disjoint family such that $\fu(X)\subseteq A_i$, we must have that 
 $\lfloor\log_2|\cdot|\rfloor:X\longrightarrow\omega$ is an injective function, 
 and hence $X$ must be countable.
\end{proof}

An easy consequence of Theorem~\ref{nouncountablehindman} is that many other 
Ramsey-theoretic results will not have an uncountable analog either. Let us consider 
two of these. For instance, Gower's $c_0$ theorem and Carlson's theorem 
on sequences of variable words\footnote{This result of Carlson~\cite{carlson}, also 
discovered independently by Furstenberg and Katznelson~\cite{furstenberg-katznelson}, is sometimes 
(e.g. in~\cite[Theorem 2.35]{ramseysp}) referred to as \textit{the infinite 
Hales-Jewett theorem}.}  are extensions of results that have Hindman's 
theorem as a straightforward consequence. Thus, any reasonable statement 
constituting an analog 
of these results to the uncountable realm should have $\hind{\kappa,\lambda}$ 
as an easy consequence, for some uncountable $\lambda$ and a $\kappa\geq\lambda$; and 
will therefore fail because of Theorem~\ref{nouncountablehindman}.

As a final observation, I would like to anticipate that in a forthcoming joint paper 
between Chodounsk\'y, Krautzberger and the author, the partition from the 
proof of Theorem~\ref{nouncountablehindman} is used (with $\kappa=\omega$) to 
show that the core of a union ultrafilter is not a rapid filter (where an ultrafilter 
$u$ on $[\mathbb N]^{<\aleph_0}$ is said to be a union ultrafilter if it has a base of 
sets of the form $\fu(X)$ for an infinite pairwise disjoint $X\subseteq[\mathbb N]^{<\aleph_0}$, 
and the core of $u$ is $\{\bigcup A\big|A\in u\}$, which is a filter on $\mathbb N$).

\section{Generalizations in terms of Abelian Groups}
\label{sec:2}

If one proves Hindman's theorem~\ref{eldehindman} via the argument of Galvin and Glazer (see e.g. 
\cite[Th. 5.8]{hindmanstrauss}), one can see right away that the scope of this theorem goes way beyond 
the realm of the natural numbers $\mathbb N$. This is, that same argument yields the following much more 
general statement, which is therefore commonly known as the Galvin-Glazer-Hindman Theorem.

\begin{theorem}
 Let $S$ be any semigroup, and suppose that we partition $S=A_0\cup A_1$ into two cells. Then there 
 exists a sequence $\vec{x}=\langle x_n\big|n<\omega\rangle$ (which in most cases of interest can be 
 found to be injective) and an $i\in2$ such that 
 $\fp(\vec{x})\subseteq A_i$, where $\fp(\vec{x})$ denotes the set
\begin{equation*}
 \left\{x_{k_0}*x_{k_1}*\cdots *x_{k_l}\big|l<\omega\text{ and }k_0<k_1<\cdots<k­­_l<\omega\right\}
\end{equation*}
of all \emph{finite products} of elements from the sequence $\vec{x}$.
\end{theorem}

Thus, given a semigroup $S$ and an ordinal $\alpha$, we introduce the symbol $\hind{S,\alpha}$ to denote
the statement that whenever we partition $S=A_0\cup A_1$ into two cells, there exists an $\alpha$-sequence 
$\vec{x}=\langle x_\xi\big|\xi<\alpha\rangle$ and an $i\in 2$ such that $\fp(\vec{x})\subseteq A_i$ (where 
$\fp(\vec{x})$ consists of all finite products $x_{\xi_0}*x_{\xi_1}*\cdots *x_{\xi_l}$ such that $l<\omega$ 
and $\xi_0<\xi_1<\cdots<\xi_l<\alpha$). If our semigroup $S$ is commutative, chances are that we will be using 
additive notation and so we will use the symbol $\fs(\vec{x})$ rather than $\fp(\vec{x})$ (finite sums 
instead of finite products). Also, in this case the order in which the sums are taken is not important and 
so we will only consider the statements $\hind{S,\lambda}$ where $\lambda$ is a cardinal (since in this 
case $\hind{S,\alpha}$ holds if and only if $\hind{S,|\alpha|}$ holds). The main result of this section 
is that for every abelian group (in fact, for every commutative cancellative semigroup) $G$, and for every 
uncountable cardinal $\lambda$, the statement 
$\hind{G,\lambda}$ fails\footnote{In the notation of our recent follow-up 
paper~\cite{assaf-david}, 
the failure of this statement is denoted by $G\nrightarrow[\lambda]_2^{\fs}$.}.

\begin{theorem}\label{abeliangroup}
 Let $G$ be a commutative cancellative semigroup. Then there exists a partition $G=A_0\cup A_1$ of $G$ into 
 two cells such that for no uncountable $X\subseteq G$ and no $i\in 2$ do we have that $\fs(X)\subseteq A_i$. 
 This is, the statement $\hind{G,\aleph_1}$ (and hence also $\hind{G,\lambda}$ for every uncountable $\lambda$) 
 fails.
\end{theorem}

\begin{proof}
 Since $G$ is commutative and cancellative, it is possible to embed $G$ into $\bigoplus_{\alpha<\kappa}\mathbb T$ 
 for some cardinal $\kappa$, where $\mathbb T=\mathbb R/\mathbb Z$ is the unit circle group. Moreover, it is 
 possible to do this embedding in such a way that the $\alpha$-th projection $\pi_\alpha[G]$ is either (isomorphic 
 to) $\mathbb Q$ 
 or a quasicyclic group, so in either case the projection is a countable set (this is all 
explained with detail in 
 \cite[p. 123]{yonilaboriel}). Thus throughout this proof, every element $x\in G$ will be thought of as a 
 member of $\bigoplus_{\alpha<\kappa}\mathbb T$, with $\alpha$-th projections denoted by 
 $\pi_\alpha(x)\in\mathbb T=\mathbb R/\mathbb Q$ and 
 finite support $\supp(x)=\{\alpha<\kappa\big|\pi_\alpha(x)\neq 0\}$.
 
 In a way totally similar to what we did in the proof of Theorem~\ref{nouncountablehindman}, 
 for $i\in 2$ we define 
 \begin{equation*}
  A_i=\left\{x\in G\big|\lfloor\log_2|\supp(x)|\rfloor\equiv i\mod 2\right\}
 \end{equation*}
 and claim that $G=A_0\cup A_1$ is the 
 partition that makes the theorem work. So by way of contradiction, we start by assuming that 
 $X\subseteq G$ is an uncountable subset such that $\fs(X)\subseteq A_i$ for some $i\in 2$. We first 
 notice that for any given finite $F\subseteq\kappa$, there can only be countably many elements 
 $x\in G$ such that $\supp(x)=F$ (because we assumed that each $\pi_\alpha[G]$ is countable) and 
 so by thinning out $X$ we can assume that the supports of elements of $X$ are pairwise distinct. 
 Furthermore, by the $\Delta$-system lemma (see e.g. \cite[Th. 1.5]{kunen} or \cite[Th. 16.1]{justweese}), 
 there 
 exists an uncountable $Y\subseteq X$ such that the supports of elements from $Y$ form a 
 $\Delta$-system, this is, there is a fixed finite $R\subseteq\kappa$ (called the \textbf{root} 
 of the $\Delta$-system) such that for every 
 two distinct $x,y\in Y$, we have that $\supp(x)\cap\supp(y)=R$.
 
 Ideally, we would like to have that $\supp\left(\sum_{k<l}x_k\right)=\bigcup\limits_{k<l}\supp(x_k)$ 
 whenever $x_0,x_1,\ldots,x_l\in Y$, but in order to ensure that we still need to process $Y$ a bit 
 more.
 
 \begin{claim}\label{refinardeltasistema}
  There exists an uncountable $Z$ with $\fs(Z)\subseteq A_i$ such that the supports of its elements 
  form a $\Delta$-system and moreover, whenever $x_0,x_1,\ldots,x_l\in Z$ we have that 
  $\supp\left(\sum_{k<l}x_k\right)=\bigcup\limits_{k<l}\supp(x_k)$.
 \end{claim}
 
 \begin{proof}[Proof of Claim]
  Let $n=|R|$ with $R=\{\alpha_1,\ldots,\alpha_n\}$. We will recursively construct 
  a sequence of uncountable sets $Y_0,Y_1,\ldots,Y_{n+1}$ such that $Y=Y_0$ and each 
  $Y_{k+1}$ is 
  a \emph{sum subsystem} of $Y_k$ (this is, for each $x\in Y_{k+1}$ there is a 
  finite $F_x\subseteq Y_k$ such that $x=\sum_{y\in F_x}y$ and moreover whenever 
  $x,y\in Y_{k+1}$ we have that $F_x\cap F_y=\varnothing$), and satisfying that either 
  $\alpha_k\in\supp(x)$ for all $x\in\fs(Y_k)$, or $\alpha_k\notin\supp(x)$ for all 
  $x\in\fs(Y_k)$. 
  In the end we will let $Z=Y_{n+1}$ and 
  $r=\{\alpha_k\big|\alpha_k\in\supp(x)\text{ for all }x\in\fs(Y)\}\subseteq R$. This will 
  ensure that the supports of elements from $Z$ form a $\Delta$-system with root $r$, and 
  moreover we will have that $Z$ is a sum subsystem of $Y$, hence 
  $\fs(Z)\subseteq\fs(Y)\subseteq A_i$. Finally, the fact that $r\subseteq\supp(x)$ for every 
  $x\in\fs(Z)$ will imply that $\supp\left(\sum_{k<l}x_k\right)=\bigcup\limits_{k<l}\supp(x_k)$ 
  whenever $x_0,x_1,\ldots,x_l\in Z$.
  
  Now for the construction, suppose that we have already constructed $Y_k$ satisfying the 
  imposed requirements. To simplify notation let $\alpha=\alpha_{k+1}$. Since we assumed that 
  $\pi_\alpha[G]$ is countable, by the pigeonhole principle there is an uncountable 
  $Y'\subseteq Y_k$ such that all of the $\pi_\alpha(y)$, for $y\in Y'$, equal some fixed 
  $t\in\mathbb T$. If this $t$ is of infinite order, then we simply make $Y_{k+1}=Y'$ and notice 
  that $\alpha\in\supp(x)$ for every $x\in\fs(Y_{k+1})$. If, on the other hand, 
  $t$ is of finite order (say, of order $n$) then we partition $Y'=\bigcup\limits_{\xi<|Y'|}F_\xi$ 
  into uncountably many cells $F_\xi$ of cardinality $n$, and let $Y_{k+1}$ consist of the elements 
  $\sum_{x\in F_\xi}x$ for $\xi<|Y'|$. Then we will have that $\alpha\notin\supp(y)$ for 
  every $y\in Y_{k+1}$ and subsequently, $\alpha\notin\supp(x)$ for every $x\in\fs(Y_{k+1})$. This 
  finishes the construction.
 \end{proof}
 
 We now use the $Z$ given by the claim in order to reach a contradiction. We will argue that 
 the function $|\supp(\cdot)|\upharpoonright Z:Z\longrightarrow\omega$ is 
 finite-to-one, which will imply that $Z$ must be countable, contrary to its construction. 
 So assume that there is an infinite family $\{z_k\big|k<\omega\}\subseteq Z$ such that 
 all of the $|\supp(z_k)|$ are equal to some $l$, and let $m<\omega$ be such that 
 $2^m\leq l<2^{m+1}$ (then by assumption, $i\equiv m\mod 2$). Also, let $n$ be such that 
 $2^n\leq l-|r|<2^{n+1}$ (note that $n\leq m$). 
 We then let $z=\sum_{k\leq 2^{m-n}}z_k\in A_i$ and notice that, since by the claim we have 
 that $\supp(z)=\bigcup\limits_{k\leq 2^{m-n}}\supp(z_k)$ and the $\supp(z_k)$ form a $\Delta$-system with root 
 $r$, we can conclude that 
 \begin{align*}
  |\supp(z)|&=\left(\sum_{k\leq 2^{m-n}}|\supp(x_k)|\right)-2^{m-n}|r|=(2^{m-n}+1)l-2^{m-n}|r| \\
  &=2^{m-n}(l-|r|)+l,
 \end{align*}
 and since $2^m\leq 2^{m-n}(l-|r|)<2^{m+1}$, we conclude that $2^{m+1}\leq|\supp(x)|<2^{m+2}$, meaning 
 that $x\in A_{1-i}$, contrary to the assumption.
\end{proof}

It might be argued that the statement $\hind{\kappa,\lambda}$ as in the previous section is the ``wrong'' way of generalizing the finite-union version of Hindman's theorem, and that one should consider partitions of $[\kappa]^{<\kappa}$ 
instead of $[\kappa]^{<\omega}$. However, if we equip $[\kappa]^{<\kappa}$ with the symmetric difference 
$\bigtriangleup$ as a 
group operation, we obtain an abelian group (in fact, a Boolean group) with the peculiarity that taking 
a finite union of pairwise disjoint elements of $[\kappa]^{<\kappa}$ coincides with taking its finite 
sum according to this group operation. Hence Theorem~\ref{abeliangroup} implies that this purported 
finite-union generalization of Hindman's theorem also fails at all uncountable 
cardinals.

\section{The noncommutative case}
\label{sec:3}

The next natural question is whether the previous results can be generalized to non-commutative semigroups. 
This is, is it true that $\hind{S,\lambda}$ fails whenever $\lambda$ is uncountable and $S$ is any semigroup? 
As test cases for this question, the first two that come to mind are the free semigroup and the free group. 
If we let $S_\kappa$ be the free semigroup on $\kappa$ generators, and consider the partition 
$S_\kappa=A_0\cup A_1$ with 
\begin{equation*}
 A_i=\{x\in S_\kappa\big|\lfloor\log_2\ell(x)\rfloor\equiv i\mod 2\}
\end{equation*}
(where $\ell(x)$ denotes the length of $x$), it is easy to see (arguing as in the proof of 
Theorem~\ref{nouncountablehindman}) that $\hind{S_\kappa,\lambda}$ fails for 
every uncountable $\lambda$. In the case of the free group, a slightly more complicated argument is 
needed.

\begin{theorem}\label{freegroup}
 Let $\kappa$ be a cardinal and let $F_\kappa$ be the free group on $\kappa$ generators. Then for 
 every uncountable ordinal $\lambda$, the statement $\hind{F_\kappa,\lambda}$ fails.
\end{theorem}

\begin{proof}
 Following the general theme of this paper, we will consider the partition $F_\kappa=A_0\cup A_1$, 
 where (letting $\ell(x)$ denote the length of a reduced word $x\in F_\kappa$)
\begin{equation*}
 A_i=\{x\in F_\kappa\big|\lfloor\log_2\ell(x)\rfloor\equiv i\mod 2\},
\end{equation*}
and we will argue that no sequence $\vec{x}=\langle x_\alpha\big|\alpha<\omega_1\rangle$ of elements 
of a given $A_i$ can be such that 
$\fp(\vec{x})\subseteq A_i$, so we assume by way of contradiction that we have such a sequence 
$\vec{x}$ with $\fp(\vec{x})\subseteq A_i$. Let $L$ (with $|L|=\kappa$) be the alphabet that 
generates $F_\kappa$, and for a reduced word $x\in L_\kappa$ we define its support by
\begin{equation*}
 \supp(x)=\{a\in L\big|\text{either }a\text{ or }a^{-1}\text{ occur in }x\}
\end{equation*}
Since there are only countably many reduced words having the same fixed support, it is possible to 
thin out the sequence $\vec{x}$ so that the supports of its members are pairwise distinct, and by 
the $\Delta$-system lemma we can also assume that said supports form a $\Delta$-system, whose root 
we will denote by $r$. Now, any member $x_\alpha$ of the sequence $\vec{x}$ can be written 
as $x_\alpha=z_\alpha y_\alpha w_\alpha$ with $\supp(z_\alpha),\supp(w_\alpha)\subseteq r$ and 
such that the first and last letters of $y_\alpha$ do not belong to $r$ (it is possible that 
$z_\alpha$ or $w_\alpha$ are empty, but $y_\alpha$ has to be nonempty). A couple of applications 
of the pigeonhole principle will thin out the sequence $\vec{x}$ in such a way that all of the 
$z_\alpha$ equal some fixed $z$ and all of the $w_\alpha$ equal some fixed $w$. Let $v$ be the 
reduced word that results from multiplying $w\cdot z$ (this is, after performing all of the 
needed cancellations). Thus for $\alpha<\beta<\omega_1$ we have that 
$x_\alpha\cdot x_\beta=zy_\alpha v y_\beta w$, where the expression on the right has no 
cancellations, and so $\ell(x_\alpha\cdot x_\beta)=\ell(x_\alpha)+\ell(x_\beta)-n$ where 
$n=\ell(w)+\ell(z)-\ell(v)$, and similarly
\begin{equation*}
 l\left(x_{\alpha_0}\cdots x_{\alpha_t}\right)=\left(\sum_{j=0}^t \ell(x_{\alpha_j})\right)-tn.
\end{equation*}
Now, the pigeonhole principle implies that there is an $\omega_1$ sequence 
$\alpha_0<\alpha_1<\cdots<\alpha_\xi<\cdots$, for $\xi<\omega_1$, such that all of the lengths 
$\ell(x_{\alpha_j})$ equal some fixed number $l$. However, if we let $k=\lfloor\log_2(l)\rfloor$ and 
$m=\lfloor\log_2(l-n)\rfloor$ (notice that $l>\ell(w)+\ell(z)\geq n$ 
so it makes sense to take the latter logarithm), then we would have that (letting 
$x=\prod_{j=0}^{2^{k-m}}x_j=x_0\cdots x_{2^{k-m}}$)
\begin{equation*}
 \ell(x)=\left(\sum_{j=0}^{2^{k-m}}\ell(x_{\alpha_j})\right)-2^{k-m}n=(2^{k-m}+1)l-2^{k-m}n=2^{k-m}(l-n)+l,
\end{equation*}
thus
\begin{equation*}
 2^{k+1}=2^{k-m}2^m+2^k\leq \ell(x)<2^{k-m}2^{m+1}+2^{k+1}=2^{k+2};
\end{equation*}
so that $k+1=\lfloor\log_2\ell(x)\rfloor$ and hence $x\in A_{1-i}$, contrary to the assumption that 
$\fp(\vec{x})\subseteq A_i$. Finding this contradiction finishes the proof.
\end{proof}

It is, however, possible to find noncommutative semigroups that behave differently to the ones 
considered so far.

\begin{example}
 Let $S$ be a linearly ordered set and turn it into a semigroup by making $x*y=\max\{x,y\}$ (everything we 
 say for this example also holds if we consider $x*y=\min\{x,y\}$). This is a commutative semigroup with the 
 property that for every $X\subseteq S$, $\fs(X)=X$. Thus the statement $\hind{S,|S|}$ holds (as an easy 
 instance of the pigeonhole principle), regardless of whether $|S|$ is countable or uncountable.
\end{example}

\begin{example}\label{exordinals}
 For an ordinal $\alpha$ we let $S_\alpha$ be the semigroup of ordinals smaller than $\alpha$ with ordinal 
 addition as the semigroup operation (this semigroup is not commutative). The key observation that for 
 every infinite ordinal $\alpha$ there 
 exists a $\beta\geq\alpha$ such that $|\beta|=|\alpha|$ and $\gamma+\delta=\delta$ whenever $\delta\geq\beta$ 
 and $\gamma\leq\alpha$ (which follows by just taking $\beta=\alpha\cdot\omega$ where $\cdot$ is ordinal 
 multiplication) allows us to conclude that for every infinite ordinal $\alpha$, the statement 
 $\hind{S_\alpha,\cf(|\alpha|)}$ holds. For if we have a partition $S_\alpha=A_0\cup A_1$ into two cells, 
 by the pigeonhole principle we must have that $|A_i|=|\alpha|$ for some $i\in 2$, and consequently we 
 can recursively build a sequence $\vec{\gamma}=\langle\gamma_\xi\big|\xi<\cf(|\alpha|)\rangle$ by picking 
 a $\gamma_\xi>\left(\sup\{\gamma_\eta\big|\eta<\xi\}\right)\cdot\omega$ with $\gamma_\xi\in A_i$ and 
 $\gamma_\xi<|\alpha|$. Thus the observation at the beginning of this example implies 
 that $\gamma_{\xi_1}+\cdots+\gamma_{\xi_n}=\gamma_{\xi_n}$ whenever $\xi_1<\cdots<\xi_n<\cf(|\alpha|)$ and 
 so $\fp(\vec{\gamma})=\{\gamma_\xi\big|\xi<\cf(|\alpha|)\}\subseteq A_i$. (If one is more careful, it is 
 possible to construct this sequence with length $|\alpha|$, but that is not so relevant since we only wanted 
 to show that $\hind{S_\alpha,\kappa}$ holds for some $\alpha$ and uncountable $\kappa$.)
\end{example}

\section{Conclusions}

The main result that we have proved in this paper, is that the uncountable analog of Hindman's theorem 
in the realm of commutative cancellative semigroups fails, in the sense that any such semigroup 
$S$ can be partitioned in two cells, in such a way that for no uncountable $X\subseteq S$ is it possible for the 
set $\fs(X)$ to be contained within one single cell of the partition. An analogous result holds for 
the symmetric group as well (with the standard required changes in the definition of $\fs(X)$ to 
account for the non-commutativity of this group). As a consequence of this, when considering 
uncountable analogs of the Ramsey-theoretic 
results that have Hindman's theorem as a particular case (such as Gowers's theorem, or the infinitary 
version of the Hales-Jewett theorem), we have that these analogs fail as well.

When we drop cancellativity, we are able to obtain two examples of (non-commutative) semigroups 
for which the uncountable analog of Hindman's theorem holds. Something that both of these 
examples have in common is that the semigroups $S$ under consideration 
contain elements $x\in S$ that can ``absorb'' many $y\in S$ in the sense that $y*x=x$. Thus, it is conceivable 
to conjecture that $\hind{S,\lambda}$ fails for uncountable $\lambda$ provided that we are dealing with an 
$S$ that does not involve the aforementioned phenomenon, which naturally leads to the following question.

\begin{question}
 Does there exist a weakly right cancellative (or a cancellative) semigroup $S$ and an uncountable ordinal 
 $\alpha$ such that $\hind{S,\alpha}$ holds? We can also restrict our attention to groups: Does there 
 exist a (non-abelian) group $G$ and an uncountable $\alpha$ such that $\hind{G,\alpha}$ holds?
\end{question}

A partial (negative) answer to the previous question (in the context of groups) was provided by 
Milliken, who showed \cite[Th. 9]{milliken} that $\hind{G,|G|}$ fails whenever $G$ is a group satisfying that 
$|G|=\kappa^+=2^\kappa$ for some infinite cardinal $\kappa$ (in particular, assuming $\ch$ we get 
that $\hind{\mathbb R,\aleph_1}$ fails, which is \cite[Cor. 11]{milliken}). Theorem~\ref{freegroup}, and 
also the remark in the paragraph prior to that theorem, constitute partial negative answers to this question 
as well.

\section*{Acknowledgements}
The author is also grateful to Michael Hru\v{s}\'ak and 
Osvaldo Guzm\'an Gonz\'alez for fruitful discussions, and to Boaz Tsaban for making several 
useful comments about this paper.


\begin{thebibliography}{furstenberg-katznelson}

\bibitem{baumgartner}
Baumgartner, J., {\em A short proof of Hindman's theorem.}
J. Combin. Theory Ser. A \textbf{17} (1974), 384--386.

\bibitem{carlson}
Carlson, T., {\em Some unifying principles in Ramsey theory.}
Discrete Math. \textbf{68} (1988), 117--169.

\bibitem{yonilaboriel}
Fern\'andez Bret\'on, D., {\em Every strongly summable ultrafilter on $\bigoplus\mathbb Z_2$ is sparse.}
 New York J. Math. \textbf{19} (2013), 117--129.
 
 \bibitem{assaf-david}
 Fern\'andez-Bret\'on, D. and Rinot, A., {\em Strong failures of higher analogs of Hindman's theorem.}
 Preprint (arXiv:1608.01512).
 
 \bibitem{furstenberg-katznelson}
 Furstenberg, H. and Katznelson, Y., {\em Idempotents in compact semigroups and Ramsey theory.}
 Israel J. Math \textbf{68} (1989), 257--270.
 
 \bibitem{michael-osvaldo}
 Guzm\'an Gonz\'alez, O. and Hru\v{s}\'ak, M., {\em personal communication.}
 
 \bibitem{hindmanthm}
 Hindman,  N., {\em Finite sums from sequences within cells of a partition of $N$.}
 J. Combin. Theory Ser. A \textbf{17} (1974), 1--11.
 
 \bibitem{hindmanstrauss}
 Hidnman,  N. and Strauss, D., {\em Algebra in the Stone-\v{C}ech compactification}, 2nd. ed.
 Walter de Gruyter, Berlin (2012).
 
 \bibitem{justweese}
 Just, W. and Weese, M., {\em Discovering modern set theory II: Set-theoretic tools for every mathematician.}
 Graduate Studies in Mathematics vol. 18, American Mathematical Society (1995).
 
 \bibitem{kunen}
 Kunen, K., {\em Set theory: An introduction to independence proofs.}
 Studies in Logic and the Foundations of Mathematics vol. 102,  North Holland (1980).
 
 \bibitem{milliken}
 Milliken, K. R., {\em Hindman' theorem and groups.}
 J. Combin. Theory Ser. A \textbf{25} (1978), 174--180.
 
 \bibitem{ramseysp}
Todor\v{c}evi\'c, S., {\em Introduction to Ramsey spaces.}
Annals of Mathematics Studies no. 174, Princeton University Press (2010).

\bibitem{tsaban}
Tsaban, B., {\em Algebra, selections, and additive Ramsey theory.}
Preprint (arXiv:1407.7437).

\bibitem{y2zheng}
Zheng, Y. Y., {\em Selective ultrafilters on $FIN$.}
Unpublished note (available online at {\tt http://www.math.toronto.edu/yyz22/2selfin.pdf}).

\end{thebibliography}
\end{document}